\documentclass[12pt]{amsart}
\usepackage{amssymb,amsmath,amsthm}
\numberwithin{equation}{section}
\usepackage{enumerate}
\usepackage{hyperref}

\newtheorem{thm}{Theorem}[section]
\newtheorem{lemma}[thm]{Lemma}

\newtheorem{cor}[thm]{Corollary}
\newtheorem{prop}[thm]{Proposition}
\theoremstyle{definition}

\newtheorem{defn}[thm]{Definition}

\usepackage{color}
\begin{document}
\title{Factorization in weak products of complete Pick spaces}

\author{Michael T. Jury}
\address{University of Florida}
\email{mjury@ufl.edu}

\author{Robert T.W. Martin}
\address{University of Cape Town}
\email{rtwmartin@gmail.com}

\thanks{Second author acknowledges support of NRF CPRR Grant 105837.}
\date{\today}

\begin{abstract}
Let $\mathcal H$ be a reproducing kernel Hilbert space with a normalized complete Nevanilinna-Pick (CNP) kernel. We prove that if $(f_n)$ is a sequence of functions in $\mathcal H$ with $\sum\|f_n\|^2<\infty$, then there exists a contractive column multiplier $(b_n)$ of $\mathcal H$ and a cyclic vector $f\in \mathcal H$ so than $b_nf=f_n$ for all $n$. 

The space of weak products $\mathcal H\odot\mathcal H$ consists of functions of the form $h=\sum_{i=1}^\infty f_ig_i$ with $f_i, g_i\in\mathcal H$ and $\sum_{i=1}^\infty \|f_i\|\|g_i\|<\infty$. Using the above result, in combination with a recent result of Aleman, Hartz, McCarthy, and Richter, we show that for a large class of CNP spaces (including the Drury-Arveson spaces $H^2_d$ and the Dirichlet space in the unit disk) every $h\in\mathcal H\odot\mathcal H$ can be factored as a single product $h=fg$ with $f,g\in\mathcal H$. 
\end{abstract}

\maketitle
\section{Introduction}

A (normalized) {\em complete Nevanlinna-Pick (CNP)} kernel on a set $\Omega$ is a function $k:\Omega\times\Omega\to \mathbb C$ of the form
\begin{equation}
  k(x,y)=\frac{1}{1-u(x)u(y)^*}
\end{equation}
where $u$ is any function from $\Omega$ into the open unit ball $\mathbb B^d$ (here $d=\infty$ is allowed; in this case $u$ would be a map from $\Omega$ into $\ell^2(\mathbb N)$ satisfying $\|u(x)\|^2=\sum_{n=1}^\infty |u_n(x)|^2<1$ for all $x\in\Omega$). From a basic result of Agler and McCarthy \cite{agler-mccarthy-2000}, if we set $E=u(\Omega)\subset \mathbb B^d$ and define
\begin{equation}
  \mathcal H_E :=cl(span\{(1-zw^*)^{-1}:w\in E\})\subset H^2_d, \label{HE}
\end{equation}
then $\mathcal H(k)$ is isometrically isomorphic to $\mathcal H_E$ via the map $k(\cdot, y)\to (1-zu(y)^*)^{-1}$. Moreover the restriction map $f\to f|_E$ is a co-isometry of $H^2_d$ onto $\mathcal H_E$, and the restriction $b\to b|_E$ is a (complete) contraction of $Mult(\mathcal H^2_d)$ onto $Mult(\mathcal H_E)$.  Finally, functions in $\mathcal H_E, Mult(\mathcal H_E)$ respectively have norm-preserving extensions to functions in $H^2_d, Mult(H^2_d)$ respectively. 

\begin{thm}\label{thm:cheap-inner-outer}
  Let $\mathcal H$ be a space with a normalized CNP kernel. If $(f_n)$ is a sequence in $\mathcal H$ with $\sum_{n=1}^\infty \|f\|_{\mathcal H}^2<\infty$, then there exist a sequence $(\varphi_n)\subset Mult(\mathcal H)$ and a cyclic vector $F\in\mathcal H$, so that 
  \begin{itemize}
  \item[(i)] $\|F\|_{\mathcal H}^2\leq \sum_{n=1}^\infty \|f_n\|_{\mathcal H}^2$,
  \item[(ii)] the column
    \begin{equation*}
      \begin{pmatrix} \varphi_1 \\ \varphi_2 \\ \vdots\end{pmatrix}
    \end{equation*}
is contractive, and
\item[(iii)] $f_n=\varphi_nF$ for all $n$. 
  \end{itemize}
\end{thm}

\subsection{Weak products} For a Hilbert function space $\mathcal H$ we define the {\em weak product} $\mathcal H \odot \mathcal H$ to be 
\begin{equation}
  \mathcal H\odot \mathcal H:=\left\{ \sum_{i=1}^\infty f_i g_i : f_i, g_i\in\mathcal H, \sum_{i=1}^\infty \|f_i\| \|g_i\|<\infty \right\}.
\end{equation}
For $h\in\mathcal H\odot\mathcal H$, the quantity
\begin{equation}
  \|h\|:= \inf\left\{ \sum_{i=1}^\infty \|f_i\|_{\mathcal H}\|g_i\|_{\mathcal H}: h=\sum_{i=1}^\infty f_ig_i\right\}
\end{equation}
defines a norm on $\mathcal H\odot \mathcal H$, making it into a Banach space. Weak products were introduced by Coifman, Rochberg, and Weiss \cite{coifman-rochberg-weiss}, who proved that for the classical spaces $H^2(\partial\mathbb B^d)$ on the unit ball, we have $H^2(\partial\mathbb B^d)\odot H^2(\partial\mathbb B^d) =H^1(\partial\mathbb B^d)$ (with equivalent norms). Generically, weak products arise as the predual of the space of bounded Hankel-type bilinear forms on $\mathcal H$; we refer to \cite{AHMR-2018} for further discussion of this topic.

To state our factorization theorem, we need one more ingredient. 

\begin{defn}
We say that $Mult(\mathcal H)$ has the {\em column-row property} if: Whenever $\Phi=(\varphi_n)$ is a sequence in $Mult(\mathcal H)$ for which the column multiplication operator 
\begin{equation}
  M_\Phi^C:f\to \begin{pmatrix} \varphi_1 f\\ \varphi_2 f\\ \vdots\end{pmatrix}
\end{equation}
is bounded from $\mathcal H$ to $\mathcal H\otimes \ell^2$, the corresponding row operator
\begin{equation}
  M_\Phi^R :\begin{pmatrix} f_1 \\ f_2\\ \vdots \end{pmatrix} \to \sum_{n=1}^\infty \varphi_n f_n
\end{equation}
is bounded from $\mathcal H\otimes \ell^2$ to $\mathcal H$. We say that the column-row property holds {\em continuously with constant $c$} if for every sequence $\Phi$ we have
\begin{equation}
  \|M_\Phi^R\|\leq c\|M_\Phi^C\|.
\end{equation}
\end{defn}
A particular consequence is that if $Mult(\mathcal H)$ has the column-row property and $\Phi = (\varphi_n)$ and $\Psi=(\psi_n)$ are symbols of column multipliers, then the function
  \begin{equation}
    \chi =\sum_{n=1}^\infty \varphi_n\psi_n
  \end{equation}
belongs to $Mult(\mathcal H)$, since $M_\chi = M_\Phi^R M_\Psi^C$ is then bounded.  

For the Hardy space $H^2$ on the unit disk, the multiplier norm is the supremum norm, and of course the column-row property holds trivially, with constant $1$. In this case it is of course well known that $H^2\odot H^2=H^1$, and every $h\in H^1$ factors as $h=fg$ with $h,g\in H^2$ and $\|f\|^2_2=\|g\|^2_2=\|h\|_1$.  Trent \cite{trent-2004} proved that the column-row property holds continuously for the Dirichlet space $\mathcal D$ on the unit disk with constant $c\leq \sqrt{18}$. Very recently Aleman, Hartz, McCarthy, and Richter \cite{AHMR-2018} proved that the column-row property holds for the Drury-Arveson spaces $H^2_d$ (for {\em finite} $d$), with a constant $c_d$ depending on the dimension $d$, and in fact for a larger class of weighted Besov-type spaces in the unit ball. 

\begin{thm}\label{thm:weak-products}
 Suppose $\mathcal H$ is a Hilbert function space with normalized CNP kernel, and that $Mult(\mathcal H)$ has the column-row property. Then every $h\in\mathcal H\odot\mathcal H$ can be factored as $h=fg$ for some $f, g\in \mathcal H$.  Moreover, if the column-row property holds continuously with constant $c$, then the $f,g$ can be chosen so that $\|f\|_{\mathcal H}\|g\|_{\mathcal H}\leq c\|h\|_{\mathcal H \odot\mathcal H}$.  

In particular the latter conclusion holds for $\mathcal H=\mathcal D$ (the Dirichlet space on the unit disk) and $\mathcal H=H^2_d$ (the Drury-Arveson space), for finite $d$. 
\end{thm}

We give the proof here, since it is an immediate application of Theorem~\ref{thm:cheap-inner-outer}, which we will then prove in the next section. 

\begin{proof}[Proof of Theorem~\ref{thm:weak-products}]
  Suppose $h=\sum_{i=1}^\infty f_i g_i $ with $\sum_{i=1}^\infty \|f_i\|\|g_i\|<\infty$. For each $i$, multiplying $f_i$, $g_i$ by appropriate constants we may assume that $\|f_i\|=\|g_i\|$ for all $i$, and hence $\sum_{i=1}^\infty \|f_i\|^2 =\sum_{i=1}^\infty \|g_i\|^2<\infty$.  Applying Theorem~\ref{thm:cheap-inner-outer} there exist column contractions $(\varphi_i)$ and $(\psi_i)$ and cyclic vectors $F,G$ such that for all $i=1,2,\dots$ we have
  \begin{equation}
    f_i=\varphi_i F, \quad g_i=\psi_i F. 
  \end{equation}
By the column-row property,
the row $(\psi_1, \psi_2, \dots )$ is also bounded multiplier, and therefore 
\begin{equation}
  m=\sum_{i=1}^\infty \varphi_i \psi_i \in Mult(\mathcal H).
\end{equation}
Thus
\begin{equation}\label{eqn:factorization}
  h=\sum_{i=1}^\infty f_ig_i = mFG =fg
\end{equation}
if we put $f=mF, g=G$.  

Suppose now the column-row property holds with constant $c$. Then the $m$ in (\ref{eqn:factorization}) will have $\|m\|_{Mult(\mathcal H)}\leq c$.  Moreover, for any $\epsilon >0$ the sequences $(f_i), (g_i)$ in the above argument can be chosen so that 
\begin{equation}
  \sum_{i=1}^\infty \|f_i\|^2_{\mathcal H} = \sum_{i=1}^\infty \|g_i\|^2_{\mathcal H} <(1+\epsilon) \|h\|_{\mathcal H\odot\mathcal H}, 
\end{equation}
hence by Theorem~\ref{thm:cheap-inner-outer} the $F,G$ can be chosen with $\|F\|^2=\|G\|^2 <(1+\epsilon)\|h\|$, so that in (\ref{eqn:factorization}) we will have $\|f\|_{\mathcal H}\|g\|_{\mathcal H} <c(1+\epsilon)\|h\|$. Taking a sequence of $\epsilon$'s 
tending to $0$, we can take weak limits of the corresponding $f$ and $g$ in $\mathcal H$, to obtain a factorization obeying $\|f\|_{\mathcal H}\|g\|_{\mathcal H} \leq c\|h\|_{\mathcal H\odot\mathcal H}$. 

\end{proof}

\section{Proof of Theorem~\ref{thm:cheap-inner-outer}}

\subsection{The free semigoup algebra and free liftings}
We write $\mathbb F_d^+$ for the free semigroup on $d$ letters $\{1, 2, \dots \}$ (again, $d$ countably infinite is allowed); that is, the set of all words $\alpha= i_1i_2\cdots i_k$, over all (finite) lengths $k$, where each $i_j\in\{1, 2, \dots\}$. We write $|\alpha|=k$ for the length of $\alpha$. We also include the {\em empty word} in $\mathbb F_d^+$, and denote it $\varnothing$, and put $|\varnothing|=0$.  There is a {\em transpose map} which reverses the letters in a word: if $\alpha=i_1i_2\cdots i_k$ then we write
\begin{equation}
  \alpha^\dag = i_ki_{k-1}\cdots i_2i_1.
\end{equation}

The {\em Fock space} $\mathcal F^2_d$ is a Hilbert space with orthonormal basis $\{\xi_\alpha\}_{\alpha\in\mathbb F^+_d}$ labeled by the free semigroup on $d$ letters. For each letter $i$ there is an isometric operator $L_i$ acting in $\mathcal F^2_d$ defined on the orthonormal basis $\{\xi_\alpha\}$ by
  \begin{equation}
    L_i\xi_\alpha = \xi_{i\alpha}
  \end{equation}
(called the {\em left creation operators}), here $i\alpha$ just means the word obtained from $\alpha$ by appending the letter $i$ on the right. One analogously defines the right creation operators $R_i$. The operators $L_i$ (and the $R_i$) have orthogonal ranges, and hence obey the identity
  \begin{equation}
    L_i^*L_j=\delta_{ij}I.
  \end{equation}
The {\em free semigroup algebra} $\mathcal L_d$ is the (unital) WOT-closed algebra generated by the $L_i$. For a word $\alpha =i_1i_2\dots i_k$ we write $L_\alpha=L_{i_1}L_{i_2}\cdots L_{i_k}$. One can show that the corresponding algebra generated by the right creation operators, $\mathcal R _d$, and $\mathcal L _d$ are each other's commutants (and in fact that $\mathcal R _d$ is simply the image of $\mathcal L _d$ under conjugation by the `tranpose unitary': precisely, if $W:\mathfrak F^2_d\to \mathfrak F^2_d$ is the unitary map which acts on basis vectors as $W\xi_\alpha =\xi_{\alpha^\dag}$, then we have $WL_\alpha W^* = R_{\alpha^\dag}$).  Each element $F$ of the free semigroup algebra admits a Fourier-like expansion
\begin{equation}
  F\sim \sum_{\alpha\in\mathbb F^+_d} c_{\alpha} L_\alpha,
\end{equation}
and the Ces\`aro means of this series converge in the strong operator topology (SOT) to $F$ \cite{DPac}. 

finally We recall the connection between the free function spaces $\mathfrak F^2_d$, $\mathfrak F^\infty_d$, and the Drury-Arveson space $H^2_d$ and its multiplier algebra. 

First, given a free holomorphic function $H\in\mathfrak F^2_d$, we may of course evaluate it on a tuple of $1\times 1$ matrices $z=(z_1, z_2, \dots)$ satisfying $\sum_j |z_j|^2<1$, (i.e. a point of the open unit ball $\mathbb B^d$). The resulting holomorphic function $h(z)=H(z)$ belongs to the Drury-Arveson space on $\mathbb B^d$, and in fact this map is a co-isometry. In particular every $h\in H^2_d$ has a {\em free lift} to a free function $H\in \mathfrak F^2_d$, and there is a unique such lift preserving the norm: $\|H\|=\|h\|$. Namely, if $\mathbb{N} ^d$ denotes the additive monoid of $d-$tuples of non-negative integers, and $\mathbf{n} := (n_1 , ... , n_d ) \in \mathbb{N} ^+ _d$, set $z^\mathbf{n} := z_1 ^{n_1} \cdots z_d ^{n_d}$. Since the free semigroup $\mathbb{F} _d ^+$ is the universal monoid on $d$ generators, there is a unital semi-group epimorphism, $\lambda : \mathbb{F} _d ^+ \rightarrow \mathbb{N} _d ^+$, the \emph{letter counting map}, defined by $\lambda (\alpha) = (n _1 , ..., n_d )$, where $n_k$ is the number of times the letter $k$ appears in the word $\alpha$. Every $h \in H^2 _d$ has a Taylor series expansion (about $0$) indexed by $\mathbb{N} ^+ _d$, and the map:
$$ h(z) = \sum _{\mathbf{n} \in \mathbb{N} ^+ _d  } h_{\mathbf{n}} z^\mathbf{n} \  \mapsto \  H(Z) := \sum h_{\mathbf{n}} Z ^{\mathbf{n}}; \quad Z \in \mathcal{B} _d $$ 
defines an isometric embedding of $H^2 _d$ into $\mathfrak{F} ^2 _d$ (\emph{i.e.} $H^2 _d$ is identified with symmetric Fock space), where 
$$ Z ^{\mathbf{n}} := \sum _{\alpha | \ \lambda (\alpha ) = \mathbf{n}} Z^\alpha, $$ see \cite[Section 4]{Sha2013}. 
Likewise, the map $F\to F(z)$ is a completely contractive homomorphism from $\mathfrak F^\infty_d$ onto the multiplier algebra $Mult(H^2_d)$, (see \cite[Theorem 4.4.1, Subsection 4.9]{Sha2013} or \cite[Section 2]{DP-NP}) and again (by commutant lifting) every $f\in Mult H^2_d$ has a norm-preserving free lift to some $F\in \mathfrak F^\infty_d$ \cite{Ball2001-lift,DL2010commutant}. (Unlike the Hilbert space case, however, norm-preserving free lifts from $Mult(H^2_d)$ to $\mathfrak F^\infty_d$ may not be unique \cite[Corollary 7.1]{JMfree}.)

\begin{proof}[Proof of Theorem~\ref{thm:cheap-inner-outer}]
We first prove the theorem in the case of the Drury-Arveson space $H^2_d$, then deduce the case of general CNP kernels. 

Fix a  sequence $(f_n)\subset H^2_d$ with $\sum_{n=1}^\infty \|f_n\|^2 <\infty.$ By the above remarks we may lift each of these functions $f_n$  to an element $F_n$ of the Fock space $\mathfrak F^2_d$, with the same norm. Now we let $\mathcal M$ denote the closed $\mathcal R_d\otimes I$-invariant subspace of $\mathfrak F^2_d\otimes \ell^2$ generated by the vector
\begin{equation}
  \widetilde{F} :=\begin{pmatrix} F_1 \\ F_2 \\ \vdots \end{pmatrix}.
\end{equation}
We first examine the {\em wandering
subspace} for the restriction of $R\otimes I$ to $\mathcal M$, 
\begin{equation}
  \mathcal W = \mathcal M\ominus (R\otimes I)\mathcal M.
\end{equation}
(Here $(R\otimes I)\mathcal M$ is a shorthand for the (closed) span of the ranges of the operators $R_j\otimes I$ restricted to $\mathcal M$. )
By the Davidson--Pitts version of the Beurling theorem for $\mathcal R_d$ (stated in \cite[Theorem 2.1]{davidson-pitts-1999} for $\mathcal L_d$, but it is evident that the analogous statements hold for $\mathcal R_d$), $\mathcal M$ is equal to the closure of the span of the orthogonal family of spaces
\begin{equation}
  \bigvee_{\alpha\in\mathbb F^+_d} (R\otimes I)^\alpha \mathcal W.
\end{equation}

We claim this space $\mathcal W$ is one-dimensional. This is immediate from the fact that by construction, $\mathcal M$ is cyclic for $R\otimes I|_{\mathcal M}$, more precisely, we claim that if $P_{\mathcal W}:\mathcal M\to \mathcal W$ is the orthogonal projection, then $P_{\mathcal W} \widetilde{F}$ spans $\mathcal W$. Indeed, if $G\in \mathcal W$ and $G\bot P_{\mathcal W} \widetilde{F}$, then since $(I-P_{\mathcal W})\widetilde{F}$ is in $W^\bot$, of course $G\bot (I-P_{\mathcal W})\widetilde{F}$ also, so $G\bot \widetilde{F}$. But also, for every word $\alpha$ with length at least 1, we have $(R^\alpha\otimes I)\widetilde{F}\in W^\bot$, so $G$ is orthogonal to all of these as well. Thus $G$ is orthogonal to $\widetilde{F}$ and all of its shifts, hence $G=0$ because $\widetilde{F}$ is cyclic. 

Applying the Beurling theorem again, since $dim \mathcal W =1$ there is an isometric (left) column multiplier
\begin{equation}\label{eqn:Mphi}
  M_\varphi :H\to \begin{pmatrix} \varphi_1 H \\ \varphi_2 H \\ \vdots \end{pmatrix}
\end{equation}
taking $\mathfrak F^2_d$ onto $\mathcal M$, intertwining the actions of $R$ and $R\otimes I|_{\mathcal M}$.  In particular there exists an $F\in\mathfrak F^2_d$ so that $M_\varphi F= \widetilde F$, that is, $\varphi_n F=F_n$ for each $n$. Since $M_\varphi$ is an isometry, $\|F\|^2 = \sum_{n=1}^\infty \|F_n\|^2$.  We also observe that $F$ is an $R$-cyclic vector in $\mathfrak F^2_d$, since it is corresponds to the $R\otimes I|_{\mathcal M}$-cyclic vector $\widetilde F$ under the unitary map (\ref{eqn:Mphi}). 

If we now pass to commuting arguments $\varphi \to \varphi(z)$, $F\to F(z)$, by our preliminary remarks the  holomorphic functions $\varphi(z)$ form a column contraction on $H^2_d$, and the function $F(z)$ belongs to $H^2_d$, with $\|F\|^2_{H^2_d}\leq \|F\|^2_{\mathfrak F^2_d} = \sum_{n=1}^\infty \|F_n\|^2 = \sum_{n=1}^\infty \|f_n\|^2_{H^2_d}$, and $\varphi_nF=f_n$ for all $n$.  Finally, since $F$ was $R$-cyclic for $\mathfrak F^2_d$, its commutative image $F(z)$ is cyclic for $H^2_d$.

Now let let $k$ be a normalized CNP kernel, by our earlier remarks we identify $\mathcal H(k)$ with $\mathcal H_E\subset H^2_d$.   Given the sequence $(f_n)\subset \mathcal H_E$, (recall Equation (\ref{HE}) for the definition of $\mathcal{H} _E$), we extend these functions isometrically to $H^2_d$, invoke the result just proved, and restrict the multipliers $\varphi_n$ and the function $F$ we obtain back to $E$.  These satisfy the conclusions (i)-(iii) of the theorem; the only thing that remains to be checked is that if $F$ is cyclic for $Mult(H^2_d)$, then its restriction to $E$ is cyclic for in $Mult(\mathcal H_E)$. We prove the contrapositive. If $g\in \mathcal H_E$ and $\langle g, M_\varphi F|_E\rangle_{\mathcal H_E}=0$ for all $\varphi\in Mult(\mathcal H_E)$, then extending $g$ to $H^2_d$ we have (since the inclusion $\mathcal H_E\subset H^2_d$ is isometric) $\langle g, M_\varphi F\rangle_{H^2_d}=0$ for all $\varphi\in Mult(H^2_d)$, whence $F$ is not cyclic. 

\end{proof}

\bibliographystyle{plain} 
\bibliography{WP} 

\end{document}